\documentclass{article}
\usepackage{graphicx} 
\usepackage{amssymb}
\usepackage{amsthm}
\usepackage{amsmath}
\usepackage{mathtools}
\usepackage{bbm}

\usepackage{xcolor}

\newcommand{\br}{\mathbb{R}}
\newcommand{\bn}{\mathbb{N}}
\newcommand{\bz}{\mathbb{Z}}

\newcommand{\chfun}{\mathbbm{1}}

\newcommand{\cb}{\mathcal{B}}
\newcommand{\cf}{\mathcal{F}}
\newcommand{\cg}{\mathcal{G}}
\newcommand{\cs}{\mathcal{S}}
\newcommand{\cp}{\mathcal{P}}

\DeclarePairedDelimiter{\norm}{\lVert}{\rVert}

\newcommand{\supp}{{\operatorname{supp}}}

\newcommand{\FIN}{\mathrm{FIN}}

\newcommand{\EXH}{\mathrm{EXH}}

\newtheorem{theorem}{Theorem}
\newtheorem{remark}{Remark}

\newtheorem{lemma}{Lemma}
\newtheorem{proposition}{Proposition}
\newtheorem{example}{Example}
\newtheorem{definition}{Definition}

\title{Compactness in Combinatorial Banach Spaces and Beyond}
\author{Jacek Gulgowski\thanks{University of Warsaw,
Center of Trustworthy AI for Life Sciences - International Research Agendas Programme, Warsaw, Poland;
e-mail: \texttt{j.gulgowski@uw.edu.pl}
ORCID: \texttt{0000-0002-7706-9263}},  Paweł Sawicki\thanks{University of Gda\'nsk,
Doctoral School at the Faculty of Mathematics, Physics and Informatics, Gda\'nsk, Poland;
e-mail: \texttt{p.a.sawicki@wp.pl}
ORCID: \texttt{0009-0008-4954-0362}}}
\date{September 2026}

\begin{document}

\maketitle

\begin{abstract}
In this paper we are going to discuss compactness in combinatorial Banach spaces. After discussing the definition and structure of these spaces, we proceed to propose compactness criteria for them and for their extensions. Finally, we demonstrate their connection with theory of hypergraphs.
\end{abstract}


\medskip


\section{Introduction}

The term {\em combinatorial Banach space} was introduced by Gowers on his blog entry \cite{Gowers:blog}, but the topic was investigated by many Authors before and after this publication. One can find a very good summary of different ideas and the state of the research in Sebastian Jachimek's Ph.D. dissertation \cite{JachimekPhD}.

The precise definition of the combinatorial Banach space will be given further, but before we dive into details let us have a look on the general idea of the paper. The story starts from  a cover of $\bn$ by a family $\cf$ of finite and nonempty sets and the space (which we later denote as $\FIN(\cf)$) and all sequences $x=(x_n)_{n\in\bn}$ such that
\[
\Vert x\Vert_\cf = \sup_{F \in \cf} \sum_{n \in F} |x_n| < \infty.
\]

The combinatorial Banach spaces we are investigating in the present paper are actually the subspaces of $\FIN(\cf)$ consisting of all sequences $x\in\FIN(\cf)$ such that 
\[
\lim_{m\to\infty}\Vert P_m(x)\Vert_\cf = 0,
\]
where $P_m:\FIN(\cf)\to\FIN(\cf)$ is a linear projection to the subspaces consisting of sequences $(x_n)$ where $x_n=0$ for all $n\leq m$. This subspace will be later denoted by $\EXH(\cf)$. The precise definitions and explanations of the notation will be given later.

As we can see (cf. Theorem \ref{th:cominatorial:submeasure} below) the family $\cf$ induces a nonpathological l.s.c. submeasure $\phi$ on $\bn$
so the combinatorial Banach space $\EXH(\cf)$ (and related space $\FIN(\cf)$) are actually the special cases of the more general concept of sequence spaces generated by a submeasure, as introduced in a paper \cite{FBN}.

We can look at families $\cf$ as lying between two extremes: $\cf_\infty = \{ \{n\} : n\in\bn\}$ (consisting of all singletones) and $\cf_1 = \{ C\subset\bn : |C|<\infty\}$  (consisting of all finite subsets of $\bn$) which generate the spaces $\ell^\infty$ and $\ell^1$ respectively. As we can easily observe
\[
\ell^1 \subset \FIN(\cf) \subset \ell^\infty,
\]
and
\[
\ell^1 \subset \EXH(\cf) \subset c_0,
\]
for any family $\cf$ (and the injections are continuous).

In the paper \cite{borodulin2024zoo} Authors investigated the zoo of combinatorial Banach spaces from the perspective of the relation between the combinatorics of the family $\cf$ of finite sets covering the set of natural numbers $\bn$ and the geometric properties of Banach spaces $\EXH(\cf)$ (denoted there as $X_\cf$) generated by the family $\cf$. This is the clear inspiration of the present paper where we are going to find the description of the compact sets of $\EXH(\cf)$ and of $\FIN(\cf)$. To make things a little bit simpler we will actually give criteria for a set to be {\em relatively compact}: let us remind that the set is {\em relatively compact} if its closure is a compact set.

We are going to show that the criteria is a certain combination of known criteria for classical sequence spaces (see classical Dunford and Schwartz monography \cite[Theorem~IV.5.6, p.~260 and Theorem~IV.6.5, p.~266]{DS}, we also refer the reader to \cite{GKM} where some of them are proved by means of the technique which we are going to use in the present paper).

\begin{theorem}[{\cite[Theorem 4.1]{GKM}}]\label{th:l1}The bounded set $A\subset \ell^1$ is relatively compact if and only if 
\[
\forall_{\varepsilon>0} \exists_{N\in\bn} \forall_{x\in A} \sum_{n=N+1}^{\infty} |x_n|<\varepsilon. 
\]    
\end{theorem}

\begin{theorem}\label{th:c0}The bounded set $A\subset c_0$ is relatively compact if and only if 
\[
\forall_{\varepsilon>0} \exists_{N\in\bn} \forall_{x\in A} |x_n|<\varepsilon. 
\]    
\end{theorem}

\begin{theorem}[{\cite[Theorem 4 and Proposition 11]{GKM2}}]\label{th:old:linf}The bounded set $A\subset \ell^\infty$ is relatively compact if and only if for any $\varepsilon>0$ there exists finite family of sets $U_i\subset\bn$, $i=1,...,N$, such that $\bn = \bigcup_{i=1}^N U_i$ and
\[
\forall_{x,y\in A} \forall_{i=1,...,N} \sup_{n\in U_i}|x_n-y_n|<\varepsilon. 
\]    
\end{theorem}

Below we are going to state the compactness criteria for the spaces $\EXH(\cf)$ and $\FIN(\cf)$ partially in the spirit of the above mentioned classic theorems. Some details of the discussion on compactness will actually lead us to the interesting combinatorial questions and references.

\section{Basic concepts}

Let us start with basic details related to the notation: $\bn=\{1,2,...\}$ denotes the set of all natural numbers and 
$\bn_0 = \bn\cup \{0\}$. 
The number of elements of the set $C$ will be denoted by $|C|$. The characteristic function of the set $A$ will be denoted by $\chfun_A$. The power set of the set $X$ will be denoted by $\mathcal{P}(X)$.

For a sequence $(a_n)_{n\in\bn}$ we will denote its {\em support} (i.e. the set of natural numbers $n\in\bn$ for which $a_n\neq 0$) by
\[
\supp(a_n)=\{ n\in\bn : a_n\neq 0 \}.
\]

In the present paper the only intervals we consider are either $[-1,1]\subset\br$ or subsets of $\bn$. So anytime we write the interval $[i,j]$ where $i,j\in\bn$ we mean
\[
[i,j] = \{i,\ldots,j\} \subset \bn.
\]
Similarly we will denote $[i,+\infty)=\{i,i+1,\cdots\}\subset\bn$ and $(i,j)=\{i+1,\ldots,j-1\}\subset\bn$.

\subsection{Function spaces}

For a topological space $X$ (not necessarily compact) by $C(X)$ we will denote the space of continuous functions $f\colon X\to\br$. This space will be usually endowed with pointwise convergence topology. On the other hand by $BC(X)$ we are going to denote the space of functions $f\colon X\to\br$ which are continuous and bounded. Of course when $X$ is compact we have $BC(X)=C(X)$. Also the space $BC(X)$ will usually be endowed with uniform convergence topology. Obviously the space $BC(X)$ may also be endowed with a supremum norm, then we will denote it as $BC_\infty(X)$.

The supremum norm in any space $E$ of bounded functions or sequences will later be denoted by $\norm{x}_\infty$ for $x\in E$. This refers to the space of bounded sequences $\ell^\infty$ and the space $BC(X)$ of bounded and continuous functions $f\colon X\to\br$.

\subsection{Abstract compactness criteria}\label{sec:abstract:compactness}

In this paper we are going to refer to the technique introduced in papers \cite{GKM, GKM2}. Let us assume that the Banach space $(E,\Vert\cdot\Vert)$ is equipped with a family of seminorms $(\Vert\cdot\Vert_i)_{i\in I}$ indexed by some set $I$ satisfying the following assumptions:
\begin{enumerate}
 \item[\textbf{\textup{(A1)}}] $\norm{x}_E=\sup_{i \in I}\norm{x}_i$ for $x \in E$,
 \item[\textbf{\textup{(A2)}}] for every $i,j \in I$ there is an index $k \in I$ such that $\norm{x}_i \leq \norm{x}_k$ and $\norm{x}_j \leq \norm{x}_k$ for $x \in E$; in other words, the family $\bigl\{\norm{\cdot}_i\bigr\}_{i \in I}$ forms a directed set. 
\end{enumerate}

The following definition is used in the characterization of compact subsets of space $E$
\begin{definition}[{\cite[Definition 3.6]{GKM}}]\label{def:equinormed}
Let $(E,\norm{\cdot}_E)$ be a normed space equipped with a family $\bigl\{\norm{\cdot}_i\bigr\}_{i \in I}$ of semi-norms satisfying conditions~\textbf{\textup{(A1)}} and~\textbf{\textup{(A2)}}.
A non-empty subset $A$ of $E$ is called \emph{equinormed} (or more precisely, \emph{equinormed with respect to the family $\bigl\{\norm{\cdot}_i\bigr\}_{i \in I}$}) if for every $\varepsilon>0$ there exists an index $k \in I$ such that $\norm{x}_E\leq \varepsilon + \norm{x}_k$ for all $x \in A$.
\end{definition}

We have the following theorem
\begin{theorem}[{\cite[Theorem 3.19]{GKM}}]\label{thm:compactness_ver2}
Let $(E,\norm{\cdot}_E)$ be a normed space equipped with a family $\bigl\{\norm{\cdot}_i\bigr\}_{i \in I}$ of semi-norms satisfying conditions~\textbf{\textup{(A1)}} and~\textbf{\textup{(A2)}}. Furthermore, assume that 
\begin{enumerate}
 \item[\textbf{\textup{(A4)}}] for each bounded sequence $(x_n)_{n \in \mathbb N}$ of elements of $E$ there is a subsequence $(x_{n_k})_{k \in \mathbb N}$ which is Cauchy with respect to each semi-norm $\norm{\cdot}_i$.\footnote{We keep the numbering of conditions the same is in \cite{GKM}}
\end{enumerate}
Then, a non-empty subset $A$ of $E$ is precompact if and only if it is bounded and the set $A-A$ is equinormed.
\end{theorem}

\begin{remark}
For a sequence space where all seminorms depend on finitely many terms in the sequence (i.e. each seminorm depends on finitely dimensional subspace only) and there are countably many seminorms, the condition $\textbf{\textup{(A4)}}$ is automatically satisfied via the diagonal argument.    
\end{remark}

\subsection{Families of of finite subsets covering $\bn$}

We will be working with families $\cf$ consisting of finite subsets $F\subset\bn$ and such that
$\bn=\bigcup_{F\in \cf} F$. In our investigations we will refer to two extreme (in a sense) cases very often:
\[
\cf_{\leq 1} = \{ F\subset \bn : |F|=1 \}
\]
and
\[
\cf_{<\infty}=\{ F\subset\bn : |F|<\infty \}.
\]
The first one is (in a way) the smallest one, the second one is the biggest one.

\begin{definition}
    The family $\cf$ is called {\em hereditary} if for any $F\in\cf$ and $G\subset F$ there is $G\in\cf$.
\end{definition}

As we can see each hereditary family $\cf$ satisfies
\[
\cf_{\leq 1} \subset \cf \subset \cf_{<\infty}.
\]

For any family $\cf$ (not necessarily hereditary) we may consider its {\em hereditary closure} i.e. the family $\cb$ consisting of all subsets of sets from $\cf$. This closure is surely a hereditary family and it is the smallest hereditary family consisting all sets from $\cf$.

Let us also have a closer look to another important property of considered families i.e. the {\em compactness} of the family $\cf$. Before we proceed further, let us observe that we may look at the family $\cf$ as the set of characteristic functions $\chfun_F\colon\bn\to\{0,1\}$. Of course the correspondence between subsets $F\subset\bn$ and characteristic functions $\chfun_F$ is naturally bijective. 

\begin{definition}
    We call the family $\cf$ {\em compact} if the set of characteristic functions $\{\chfun_F : F\in\cf\}\subset C(\bn)$ is compact in the pointwise convergence topology.
\end{definition}

One can easily see that out of the two extreme families mentioned above the family $\cf_{\leq 1}$ is compact while $\cf_{<\infty}$ is not.

The next Lemma is an easy consequence of the definition and is actually a folclore:
\begin{lemma}
The family $\cf$ is compact if and only if the closure of the set $\{\chfun_F : F\in\cf\}$ in $C(\bn)$ with pointise convergence topology consists of characteristic functions of finite sets only.
\end{lemma}

\subsection{Submeasure sequence spaces and combinatorial Banach spaces}

First we are going to remind concepts of sequence spaces generated by a submaesure defined on $\bn$ and introduced by Borodulin\--Nadzieja and Farkas in \cite{FBN}. Let us frst remind that a function  $\phi:\mathcal{P}(\bn)\to[0,\infty]$ is called a \emph{submeasure} if $\phi(\emptyset)=0$, $\phi(\{n\})<\infty$ for every $n\in\bn$, and 
\[
\phi(C)\leq\phi(C\cup D)\leq\phi(C)+\phi(D)
\]
for all $C,D\subseteq\bn$. A submeasure $\phi$ is \emph{lower semicontinuous} (lsc) if $\phi(C)=\lim_{n\to\infty}\phi(C\cap\{1,2,\ldots,n\})$ for each $C\subseteq\bn$.

An lsc submeasure $\phi$ is \emph{non-pathological}, if
$$\phi(C)=\sup\{\mu(C):\ \mu\text{ is a measure such that }\mu\leq\phi\},$$
for all $C\subseteq\bn$. Not every lsc submeasure is non-pathological -- see \cite[Section~1.9]{farah-book}, \cite[Theorem~4.12]{ft-mazur}, \cite[Section~6.2]{marciszewski-sobota}, \cite{MezaPat} or \cite[Theorem~4.7]{tryba-gendensity} for such examples. 

Given an lsc submeasure $\phi$, let $\mathcal{M}_\phi$ be the family of all measures $\mu$ on $\bn$ such that $\mu\leq\phi$. Then, by definition, $\phi$ is non-pathological if and only if $\phi(C)=\sup\{\mu(C): \mu\in\mathcal{M}_\phi\}$ for all $C\subseteq\bn$.

For lsc submeasure $\phi$ we may naturally define two sequence spaces (cf. \cite{FBN}).

\begin{definition}
Let $\phi$ be a non-pathological lsc submeasure. Define a function $\hat\phi:\br^\bn\to[0,\infty]$ by:
\[
\hat\phi(x)=\sup\left\{\sum_{n\in\bn}\mu(\{n\})|x_n|: \mu\in\mathcal{M}_\phi\right\}
\]
for all $x=(x_n)_{n\in\bn}\in\br^\bn$. Define also:
\[
\FIN(\phi)=\left\{x\in\br^\bn: \hat\phi(x)<\infty\right\};
\]
\[
\EXH(\phi)=\left\{x\in\br^\bn: \lim_{n\to\infty}\hat\phi(x\cdot\chi_{[n+1, +\infty)})=0\right\} =\left\{x\in\br^\bn: \lim_{n\to\infty}\hat\phi(P_n(x))=0\right\}.
\]
\end{definition}

Let us now move to the definition of {\em combinatorial Banach spaces}.

\begin{definition}
Letting $\mathcal{F}$ be a family of finite sets covering $\bn$, it is possible to define a norm on $\br^\bn$ as follows: $$\Vert x\Vert_\mathcal{F}=\sup_{F \in \mathcal{F}} \sum_{n \in F} |x_n|.$$ Then the space $$X_\mathcal{F}=\{x \in \br^\bn: \, \lim_{n \to \infty} \Vert x \cdot \chfun_{[n+1, +\infty)}\Vert_\mathcal{F}=0\}.$$ is called a \emph{combinatorial Banach space}.
\end{definition}

It is a quick observation that if we have two families $\cf,\cg$ of finite sets covering $\bn$ and such that their hereditary closure is the same then  $X_\cf=X_\cg$.

\begin{theorem}\label{th:cominatorial:submeasure}
    Combinatorial Banach spaces are a subclass of submeasure sequence spaces.
\end{theorem}
\begin{proof}
Let us define a submeasure $\Phi_\mathcal{F}$ as $\Phi_\mathcal{F}(C)=\sup_{F \in \mathcal{F}}|C\cap F|$. It is trivial to check that this is indeed a submeasure and that it is lsc.
Moreover, for any $F\in\cf$ we have a measure $\mu_F$ defined on $\mathcal{P}(\bn)$
by $\mu_F(C) = |C\cap F|$. It is easy to observe that $\mu_F(C)\leq\Phi_\cf(C)$ for all $C\subset\bn$ and $F\in\cf$ and that
\[
\Phi_{\cf}(C) = \sup_{F\in\cf}\mu_F(C).
\]
This shows that $\Phi_\cf$ is nonpathological. Moreover, it is obvoius that for any sequence $x\colon\bn\to\br$ there is $\hat\Phi_\cf(x) = \Vert x\Vert_\cf$. That's why  $X_\cf=\EXH(\Phi_\cf)$
\end{proof}

\begin{remark}
Since now we will denote the combinatorial Banach space $X_\cf$ by $\EXH(\cf)$ (instead of $\EXH(\Phi_\cf)$. Consequently we will write $\FIN(\cf)$ instead of $\FIN(\Phi_\cf)$.    
\end{remark}

It is proven in \cite[Theorem 5.4]{FBN} that, for any submeasure $\Phi$, $EXH(\Phi)=FIN(\Phi)$ if and only if $FIN(\Phi)$ is separable. The spaces $X_\cf$  generated by a family $\cf$ of finite subsets of $\bn$ covering $\bn$ is also mentioned in Section 5 in \cite{FBN}, without mentioning a term {\em combinatorial Banach space}.

\section{Function space representation of combinatorial Banach spaces}\label{sec:function:space}


In this section we are going to show that combinatorial Banach spaces may be represented as subspaces of space $C(K)$ of continuous functions on certain topological space $K$. The space $K$ is compact if and only if the family $\cf$ generating the combinatorial Banach space is compact. The idea for a definition of the space $K$ mentioned here was taken from the paper \cite{KKLPB}.

First we are going to construct a compact space $\tilde K$ being the product of countably many finite sets $D_n$, $n\in\bn$ (each of them with discrete topology) with a product topology (i.e. with a pointwise convergence topology). To make the construction more readable we are going to follow  paper \cite{KKLPB} mentioned above but turn into less general setting (instead of working with general Banach space ${\mathbf Z}$ we assume the standard situation i.e. ${\mathbf Z}=\ell^1$). Hence, let us take $D_n=\{k\cdot 4^{-n}: k=-4^{n},...,4^{n}, k\in\bz\}\subset[-1,1]$ and denote $\tilde K = \prod_{n\in\bn} D_n$. 

Let us now define the subspace $K\subset \tilde K$ given by
\begin{equation}\label{def:space:K}
    K=\left\{ \xi = (\xi_n)_{n\in\bn} \in \tilde K\cap c_{00} :\supp(\xi) \in\cf \right\}.
\end{equation}

Please note that, by definition of $\tilde K$, we know that if $\xi\in K$, then $\xi\in\ell^\infty$ and $\norm{\xi}_\infty\leq 1$.

\begin{remark}
    The topological spaces  $K_\cf$ and $K_\cg$ constructed for different families $\cf,\cg$ of finite sets covering $\bn$ are formally different, even when $\cf$ and $\cg$ have the same hereditary closure. 
\end{remark}

The following Lemma slightly extends the observation stated in the proof of (1) in Lemma 3.3. in \cite{KKLPB}.
\begin{lemma}
Assume $\cf$ is hereditary. Then the set $K$ is compact if and only if the family $\cf$ is compact.    
\end{lemma}
\begin{proof}
The proof of compactness of $K$ for compact $\cf$ is given in the mentioned paper \cite{KKLPB} but we will state it here for completeness (in much simpler version adjusted to the present context). And we will start with a proof in this direction: assume that $\cf$ is compact and take a sequence $(\xi_n)\subset K$ convergent to some $\eta\in\tilde K$. We are going to show that $\eta\in K$. Let us fix any $k\in\bn$ and take $\delta < 4^{-k}$. We know that for $n\in\bn$ big enough we have for all $j\in[1,k]$
\[
|\eta_j-\xi_{n,j}|<\delta
\]
what, taking into account that either  $\xi_{n,j}=0$ (when $j\not\in\supp(\xi_n)$) or $|\xi_n,j|>\delta$ (when $j\in\supp(\xi_n)$), means that $\eta_j=0$ iff $\xi_{n,j}=0$ for $n$ big enough and for all $j\in\bn$. Hence $\supp(\eta)\cap [1,k] =\supp(\xi_n)\cap[1,k]$. But we know that $\supp(\xi_n)\in\cf$ so, because $\cf$ is hereditary, also $\supp(\xi_n)\cap[1,k]\in\cf$. This means that for each $k\in\bn$ the set $\supp(\eta)\cap[1,k]\in\cf$. It is easy to observe that the characteristic function $\chfun_{\supp(\eta)}$ is a pointwise limit of the sequence of characteristic functions $\chfun_{\supp(\eta)\cap[1,n]}$, meaning that $\supp(\eta)$ belongs to the closure of the family $\cf$. Since $\cf$ is compact the set $\supp(\eta)$ must be finite. This means that 
$\eta\in K$, and thus $K$ is closed and compact as a subset of a compact space $\tilde K$.

Now let us assume $\cf$ is not compact, so there exists such $G\subset\bn$ that $|G|=\infty$ and $\chfun_G(k) = \lim_{n\to+\infty}\chfun_{F_n}(k)$ for all $k\in\bn$ and for a certain sequence $(F_n)_{n\in\bn}\subset\cf$. Then the sequences $\chfun_{F_n}\in K$, the sequence $\chfun_G$ is a pointwise limit of $(\chfun_{F_n})$ in $\tilde K$ and  $\chfun_G\not\in K$. Hence, $K$ is not compact. 
\end{proof}

\begin{lemma}\label{exh:fun:representation}
Assume now we have a family $\cf$ of finite sets covering $\bn$ (we neither assume that it is compact nor even hereditary).
Then there exists a subspace of a linear space $BC(K)$ of bounded and continuous functions $f\colon K\to\br$ with pointwise convergence topology, which is isometric to the space $\EXH(\cf)$.
\end{lemma}
\begin{proof}
    
We will prove this Lemma in a series of steps:

\begin{enumerate}
\item 
Let us first define the map $\Phi\colon c_{00}\to BC(K)$ by
\[
\Phi(a)(\xi) = \sum_{n\in\bn} a_n\xi_n
\]
for $a\in c_{00}$ and $\xi\in K$.  Since $\supp(\xi)$ is finite for any $\xi\in K$, the sum is well defined for any sequence $a\in\br^\bn$. We need to prove that $\Phi$ indeed falls into $BC(K)$.

First, we will show that for any $a\in c_{00}$ the function $\Phi(a):K\to\br$ is continuous. Let us take any sequence $\xi_n\to\xi$ in $K$. Since $\xi\in K$, we know that  $\supp(\xi)\subset [1,k]$ for some $k\in\bn$ and $\supp(\xi)\in\cf$. We also know that $\supp(a)\subset[1,j]$ for some $j\in\bn$. We can also observe that for $n\in\bn$ big enough the sequence $\xi_n$ is constant on the interval $[1,\max\{j,k\}]$ meaning that
\[
\lim_{n\to+\infty} \Phi(a)(\xi_n) = \lim_{n\to+\infty} \sum_{i=1}^{\max\{j,k\}} a_i\xi_{n,i} = \sum_{i=1}^{\max\{j,k\}} a_i\xi_{i} = \Phi(a)(\xi).
\]

We should also observe that $\Phi(a)\colon K\to\br$ is a bounded function. Let us take any $a\in c_{00}$ and any $\xi\in K$. Then 
\[
|\Phi(a)(\xi)| = \left| \sum_{i\in\bn} a_i\xi_i \right| \leq \sum_{i\in\supp(\xi)} |a_i|\cdot|\xi_i|  \leq \norm{a}_\cf,
\]
because $\supp(\xi)\in\cf$ and $\norm{\xi}_\infty\leq 1$.
As we can see, this estimate works also for any $a\in \FIN(\cf)$ (so also for any $a\in\EXH(\cf)$).

This means we have a map $\Phi:c_{00}\to BC(K)$.
\item The map $\Phi:c_{00}\to BC(K)$ is obviously linear, now we are going to show that it is continuous when $c_{00}$ is endowed with a norm $\norm{\cdot}_\cf$. Let us take such $a\in c_{00}$ that $\norm{a}_\cf\leq 1$. Then 
\[
\norm{\Phi(a)}_{\infty} = \sup_{\xi\in K} |\Phi(a)(\xi)|.
\]

Let us fix $\xi\in K$ and estimate
\[
|\Phi(a)(\xi)| = \left| \sum_{i\in\supp(\xi)} a_i\xi_i \right| \leq
\left(\sup_{i\in\supp(\xi)}|\xi_i|\right) \cdot \left( \sum_{i\in\supp(\xi)} |a_i| \right) \leq \Vert a\Vert_\cf
\]
since $|\xi_i|\leq 1$ for all $i\in\bn$ and $\supp(\xi)\in\cf$. This means that $\Phi:c_{00}\to BC(K)$ is linear and continuous with a norm $\Vert \Phi\Vert\leq 1$.

\item We can see that $c_{00}$ is dense in $\EXH(\cf)$ so for each $a\in\EXH(\cf)$ there exists a sequence $a_n \in c_{00}$ such that $\norm{a_n-a}_\cf\to 0$. Moreover, for any fixed $\xi\in K$ we have $\supp(\xi)\in\cf$ being a finite set, what means that
\[
\Phi(a)(\xi) = \lim_{n\to +\infty} \Phi(a_n)(\xi)
\]
is well defined. Moreover we may take $a_n$ in such a way that $\norm{a_n}_\cf \leq \norm{a}_\cf$ (e.g. by taking $a_n=a\cdot \chfun_{[1,n]}$). Then we can see that for $n$ big enough we have $\supp(\xi)\subset[1,n]$, $a_k=a_{n,k}$ for $k\in\supp(\xi)$ and consequently 
\[
\left|\Phi(a)(\xi)\right| = \left| \sum_{i\in\supp(\xi)} a_i\xi_i \right| =  \left| \sum_{i\in\supp(\xi)} a_{n,i}\xi_i \right| \leq \norm{a_n}_\cf \leq \norm{a}_\cf.
\]
Hence $\Vert \Phi(a)\Vert_{\infty} \leq \norm{a}_\cf$ for any $a\in\EXH(\cf)$.

\item Let us now prove that $\Vert \Phi(a)\Vert_{\infty} \geq \norm{a}_\cf$ for any $a\in\EXH(\cf)$. Let us take any $\varepsilon\in(0,1)$ and $F\in\cf$ that $\sum_{j\in F}|a_j|\geq (1-\varepsilon)\norm{a}_\cf$. Then 
we may take such $\eta\in K$ that
\[
\eta_n = \begin{cases} 1 & n\in F \wedge a_n>0\cr
-1 & n\in F \wedge a_n<0\cr
0 & \text{otherwise}
\end{cases}
\]
Then we have $\eta\in K$ and
\[
\Vert\Phi(a)\Vert_\infty \geq \left|\sum_{n\in\bn} a_n\eta_n\right| = \sum_{n\in F}|a_n|\geq (1-\varepsilon)\norm{a}_\cf.
\]
Since the inequality given above is correct for any $\varepsilon(0,1)$ we can see that $\Vert \Phi(a)\Vert_{\infty} \geq \norm{a}_\cf$.

\item Since $\Phi$ is an isometry it is invertible on its range  and this observation completes the proof.
\end{enumerate}

\end{proof}

\begin{remark}
    The main obstacle when trying to repeat this construction for the space $\FIN(\cf)$ is that if $\FIN(\cf)\neq \EXH(\cf)$, then the space $\FIN(\cf)$ is not separable (see Theorem 5.4 in \cite{FBN}).
\end{remark}

\begin{remark}\label{rem:not:continuous}
    The proof of Lemma \ref{exh:fun:representation} may not be easily repeated for the space $\FIN(\cf)$. Moreover, it appears that the map $\Phi$ is not necessarily defined to the space $BC(K)$. We can easily see that for any $a\in\FIN(\cf)$ the function $\Phi(a)\colon K\to\br$ is well defined and bounded, but it appears it is not necessarily continuous. Let us take a  family $\cf$ containing sets in the form $F_n=\{1,n\}$ for $n\in\bn$ and such that the cardinality of sets belonging to $\cf$ is bounded by some constant. Then the constant sequence $a_n=1$ belongs to $\FIN(\cf)$, but the function $\Phi(a)$ is not continuous. To see this, let us take the sequence $(\xi_n)\subset K$, $\xi_{n,k}=\chfun_{F_n}(k)$. As we can see $\xi_n\to\xi$, where $\xi_k = \chfun_{\{1\}}(k)$. But $\Phi(a)(\xi_n)=2\not\to \Phi(a)(\xi)=1$. 
\end{remark}

Following Theorem 17 from \cite{GKM2} we can write the following compactness criteria expressed in the language of  the spaces of bounded continuous functions on topological space $K$ (not necessarily compact (see \cite{DS}, \cite{GKM2}):
\begin{theorem}\label{compact:bc}
 The set $A\subset \EXH(\cf)$ (i.e. $\Phi(A)\subset BC(K)$) is relatively compact iff 
 \begin{enumerate}
\item[(\ref{compact:bc}.1)] for each $\eta\in K$ the set $\{\Phi(a)(\eta) : a\in A\}\subset\br$ is bounded;
     \item[(\ref{compact:bc}.2)] for any $\varepsilon>0$ there exists the finite partition of the domain    $K=\bigcup_{i=1}^N K_i$ consisting of open sets $K_i\subset K$ such that 
 \[
 \forall_{\xi,\eta\in K_i} \forall_{a\in A} |\Phi(a)(\xi)-\Phi(a)(\eta)|<\varepsilon.
 \]
 \end{enumerate} 
\end{theorem}

So we can see the compactness is expressed in terms of the properties of functions $\Phi(a)$: the functions are {\em almost constant} on certain sets uniformly with respect to $a\in A$. Now we will be trying to translate these conditions to the properties of the set $A\subset\EXH(\cf)$.

\begin{remark}
    The condition $(\ref{compact:bc}.1)$ in Theorem \ref{compact:bc} is satisfied  if and only if $A\subset \EXH(\cf)$ is bounded.
\end{remark}

\begin{remark}\label{rem:uniform:limit}The condition $(\ref{compact:bc}.2)$ in Theorem \ref{compact:bc} states that all of the sets $K_i$ are open in a product topology. This means that there exists such $n_0\in\bn$ that for all $n\geq n_0$ and for $i\in\{1,...,N\}$ all the projections of $K_i$ into $n$-th coordinate equals to $D_n$. 
Hence, assume now that for a fixed $\varepsilon>0$ condition (\ref{compact:bc}.2) is satisfied.
This in turn implies that for a fixed $\varepsilon>0$ there exists such $n_0\in\bn$ that for any fixed $i\in\{1,...,N\}$, any $a\in A$ and any $F\in\cf$ such that $F\cap[n_0+1,\infty)\neq\emptyset$ we may take $\xi,\eta\in K$ such that $\supp\xi = \supp\eta = F$, $\xi_n=\eta_n=1$ for $n\in F\cap\{1,...,n_0\}$ and
\[
\xi_n = \begin{cases}
    1 & a_n \geq 0 \cr -1 & a_n < 0,
\end{cases}
\]
and $\eta_n=-\xi_n$. We can see that since $\xi$ and $\eta$ are identical on first $n_0$ coordinates, then they belong to the same set $K_i$. Then we have 
\[
\varepsilon > |\Phi(a)(\xi)-\Phi(a)(\eta)| = \left| \sum_{n\in F} a_n(\xi_n-\eta_n) \right| = 
2\sum_{n\in F\cap[n_0+1,+\infty)} |a_n|.
\]
Hence we may observe that if the condition $(\ref{compact:bc}.2)$ is satisfied then the limit 
\[
\lim_{n \to \infty} \Vert a \cdot \chfun_{[n+1,+\infty)\}}\Vert_{\cf} = 0
\]
is uniform with respect to elements $a\in A$.
\end{remark}

The last remark will be clarified in the next section.

\section{Compactness in combinatorial Banach spaces}

The Remark \ref{rem:uniform:limit} stated at the end of the previous section appears to be a good starting point for the compactness criteria in any combinatorial Banach space.

\begin{theorem}\label{th:compact:exh}
The following conditions are equivalent for any bounded set $A\subset \EXH(\cf)$: \begin{itemize}
    \item[(\ref{th:compact:exh}.1)] the set $A$ is relatively compact;
    \item[(\ref{th:compact:exh}.2)] for any $\varepsilon>0$ there exists such $n\in\bn$ that 
    \[
    \left\Vert a\cdot\chfun_{[n+1,+\infty)}\right\Vert_\cf = \sup_{F\in\cf}\sum_{k\in [n+1,+\infty)\cap F} |a_k|<\varepsilon,
    \]
    for all $a\in A$.
\end{itemize}
\end{theorem}
\begin{proof}
 Obviously the statement that $A\subset \EXH(\cf)$ is relatively compact is equivalent to saying that $\Phi(A)\subset BC(K)$ is relatively compact. Since for bounded $A$ the set $\Phi(A)(K)\subset\br$ is automatically bounded it is enough to focus on proving the conditions $(\ref{compact:bc}.2)$ and (\ref{th:compact:exh}.2) are equivalent.   

 If $\Phi(A)\subset BC(K)$ is relatively compact, then by Theorem \ref{compact:bc} the condition (\ref{compact:bc}.2) is satisfied, and repeating the reasoning given in  Remark \ref{rem:uniform:limit} we can see that condition $(\ref{th:compact:exh}.2)$ is satisfied as well.

 On the other hand, let us assume that condition $(\ref{th:compact:exh}.2)$ is satisfied for $\varepsilon/2$. Let us denote the projection from $K$ to the product $\prod_{i=1}^n D_i$ by $\Pi_n\colon K\to \prod_{i=1}^n D_i$ and let us take the set $\Xi\subset \prod_{i=1}^n D_i$ given by $\Xi = \Pi_n(K)$. The set $\Xi$ is finite and we may use its elements to build the open cover $K_{\hat\xi}$ of $K$ and indexed by elements $\hat\xi\in\Xi$. Let us take
 \[
 K_{\hat\xi} = \{ \xi\in K : \xi_k = \hat\xi(k) {\text{\ for\ }} k\in[1,n] \} = \{\hat\xi\}\times\Pi_{i=n+1}^\infty D_i.
 \]
 There are finitely many such sets, all of them are open and their sum covers entire space $K$. Now let us take any two sequences $\xi,\eta\in K_{\hat\xi}$ and any $a\in A$. Then we have
 \[
 |\Phi(a)(\xi)-\Phi(a)(\eta)| = \left| \sum_{k=1}^n (\xi_k-\eta_k)a_k + \sum_{k=n+1}^{+\infty} (\xi_k-\eta_k)a_k \right| \leq 
 \]
 \[
 \sum_{k=n+1}^{+\infty} |\xi_k| |a_k| + \sum_{k=n+1}^{+\infty} |\eta_k| |a_k| \leq
 \sum_{k\in\supp(\xi)\cap[n+1,\infty)} |a_k| + \sum_{k\in\supp(\eta)\cap[n+1,\infty)} |a_k| \leq 2\cdot \frac{\varepsilon}{2},
 \]
 what completes the proof.
\end{proof}

Let us also observe that this compactness criterion may be concluded directly from Theorem \ref{thm:compactness_ver2}.

Let us start with a definition of a seminorm related to a family $\cb\subset \cf$:
\[
\Vert x\Vert_{\cb} = \sup_{F\in \cb} \sum_{n\in F} |x_n|.
\]

Now we are ready to define  a family of seminorms $(\Vert\cdot\Vert_{\cb})_{\cb\subset\cf\wedge |\cb|<\infty}$ in $\FIN(\phi)$ indexed by a (countable) family of finite subsets of $\cf$.
The family obviously satisfies conditions \textbf{\textup{(A1)}}, \textbf{\textup{(A2)}} and \textbf{\textup{(A4)}}. From now on, whenever we refer to an equinormed set, it means it is equinormed with respect to this family of seminorms.

What remains to be proven is the following Lemma:
\begin{lemma}\label{lem:exh:compact:cond}
Assume the family $\cf$ is hereditary. The following conditions are equivalent for a bounded set $A\subset\EXH(\cf)$:
\begin{itemize}
\item[(i)] the set $A-A$ is equinormed;
\item[(ii)] for any $\varepsilon>0$ there exists such $n\in\bn$ that 
    \[
    \sum_{k\in [n+1,+\infty)\cap F} |a_k|<\varepsilon,
    \]
    for all $F\in\cf$ and all $a\in A$.    
\end{itemize}
\end{lemma}
\begin{proof}
Let us first prove the implication $(ii)\Rightarrow (i)$. Let us first take any $\varepsilon>0$ and assume the condition $(ii)$ is satisfied for $\varepsilon/3$ and fix $n\in\bn$.  Let us define the seminorm induced by a set of all nonempty subsets of $\bn$ being the intersections $F\cap[1,n]$, where $F\in\cf$ i.e. 
\[
\cb = \{ F\cap [1,n] : F\in\cf \wedge F\cap[1,n]\neq\emptyset \}.
\]
Let us take any $a,b\in A$. There exists such $F\in\cf$ that 
\[
\sum_{k\in F} |a_k-b_k|\geq \Vert a-b\Vert_\cf -\frac{\varepsilon}{3}.
\]
On the other hand we have
\[
\Vert a-b\Vert_{\cb} \geq \sum_{k\in F\cap[1,n]}|a_k-b_k| =   
\sum_{k\in F}|a_k-b_k| - \sum_{k\in F\cap[n+1,\infty)}|a_k-b_k| \geq 
\]
\[
\sum_{k\in F}|a_k-b_k|  - \sum_{k\in F\cap[n+1,\infty)}|a_k|  - \sum_{k\in F\cap[n+1,\infty)}|b_k| \geq \Vert a-b\Vert_\cf -\frac{\varepsilon}{3} -\frac{\varepsilon}{3}-\frac{\varepsilon}{3},
\]
hence $A-A$ is equinormed.

Let us now proceed to the proof of the implication $(i)\Rightarrow (ii)$. We will show that if we assume, contrary to our claim, that $(ii)$ is not satisfied, then the set $A$ may not be relatively compact. This, in turn, will contradict that $A-A$ is equinormed. So, if $(ii)$ is not satisfied then, for certain $\delta>0$ and  for each $n\in\bn$ there exist $F_n\in\cf$ and $a_n\in A$ such that 
\[
\sum_{k\in[n+1,\infty)\cap F_n} |a_{n,k}|>\delta.
\]
We will show that the sequence $(a_n)\subset A$ may not contain a convergent subsequence. We will show that for any $n\in\bn$ there exists $m_n>n$ such that for all $k\geq m_n$
\[
\Vert a_n-a_k\Vert_\cf \geq \frac{\delta}{2} > 0.
\]
For any fixed $n\in\bn$, since $a_n\in\EXH(\cf)$ there exists  $m_n\in\bn$ such that 
\[
\Vert a_n\cdot\chfun_{[m_n,\infty)}\Vert_{\cf} <\frac{\delta}{2}
\]
We may of course assume that $m_n>n$. Then for any $k\geq m_n$ we have 
\[
\Vert a_n-a_k\Vert_\cf \geq \sum_{i\in [k,\infty)\cap F_k}|a_{n,i}-a_{k,i}|\geq
\sum_{i\in [k,\infty)\cap F_k}|a_{k,i}|- \sum_{i\in [k,\infty)\cap F_k}|a_{n,i}|
\geq \delta - \frac{\delta}{2}.
\]
This shows that $(a_n)$ does not contain any Cauchy subsequence.
\end{proof}

Theorem \ref{th:compact:exh} and Lemma \ref{lem:exh:compact:cond} leads to the following theorem.
\begin{theorem}\label{th:compact:exh:2}
Assume the family $\cf$ is hereditary. The following conditions are equivalent for any bounded set $A\subset \EXH(\cf)$: \begin{itemize}
    \item[(\ref{th:compact:exh:2}.1)] the set $A$ is relatively compact;
    \item[(\ref{th:compact:exh:2}.2)] for any $\varepsilon>0$ there exists such $n\in\bn$ that 
    \[
    \sum_{k\in [n+1,+\infty)\cap F} |a_k|<\varepsilon,
    \]
    for all $F\in\cf$ and all $a\in A$. 
\end{itemize}    
\end{theorem}
The proof is omitted as it is direct application of mentioned facts.

\section{Compactness in $\FIN$ extensions of combinatorial Banach spaces }

Let $\mathcal{F}$ be a family of finite sets in $\bn$, yielding a Banach space $\FIN(\cf)$ with a norm $\Vert\cdot\Vert_\mathcal{F}$. We wish to investigate when a set $A \subseteq \FIN(\cf)$ is compact.

\subsection{Function space representation of $\FIN(\cf)$}

First, similarly as we did in Section \ref{sec:function:space}, we are going to show that there exists a subspace of appropriate Banach function space which is isometric to $\FIN(\cf)$. Generally, it may not be the space $BC(K)$ as we observed in Remark \ref{rem:not:continuous}. 

\begin{lemma}\label{fin:fun:representation}
The space $\FIN(\cf)$ is isomorphic to the appropriate subspace of the space $B(K)$ of bounded functions $f\colon K\to\br$ endowed with a supremum norm.
\end{lemma}
\begin{proof}
We will do it in a series of steps:

\begin{enumerate}
\item 
Let us first define the map $\Phi\colon \FIN(\cf)\to B(K)$ by
\[
\Phi(a)(\xi) = \sum_{n\in\bn} a_n\xi_n = \sum_{n\in\supp(\xi)} a_n\xi_n
\]
for $a\in\FIN(\cf)$ and $\xi\in K$. Since $\supp(\xi)$ is finite for any $\xi\in K$ the sum is well defined for any sequence $a\in\FIN(\cf)$. 

We can see that for any $a\in\FIN(\cf)$ and $\xi\in K$
\[
|\Phi(a)(\xi)| = \left| \sum_{n\in\supp(\xi)} a_n\xi_n \right| \leq  \sum_{n\in\supp(\xi)} |a_n| \leq\Vert a\Vert_{\cf},
\]
what means that $\Phi(a)\in B(K)$.

\item The map $\Phi:\FIN(\cf)\to B(K)$ is obviously linear, now we are going to show that it is continuous. Let us take any $a\in \FIN(\cf)$. Then 
\[
\norm{\Phi(a)}_{\infty} = \sup_{\xi\in K} |\Phi(a)(\xi)| \leq \Vert a\Vert_\cf.
\]
This means $\Phi$ is linear and continuous with a norm $\Vert \Phi\Vert\leq 1$.

\item Similarly as in the proof of Lemma \ref{exh:fun:representation} we can see that $\Phi$ is an isometry and it is invertible (from its image). The proof is identical, so it will not be repeated.
\end{enumerate}
\end{proof}

It appears that we may embed $\FIN(\cf)$ into the space of bounded and continuous functions defined on $K$ when $K$ is endowed with a supremum metric, i.e.
\[
d_\infty(\eta,\xi) = \sup_{n\in\bn}|\eta_n-\xi_n|.
\]
This space is obviously not complete. In light of Lemma \ref{fin:fun:representation} it is enough to prove that the functions $\Phi(a):K\to \br$ are continuous. So here we go:

\begin{lemma}\label{fin:cont:function:space}Assume $K$ is endowed with $d_\infty$ metric. Then, for $a\in\FIN(\cf)$ the map $\Phi(a):K\to \br$ is continuous.
\end{lemma}
\begin{proof}
Let $\xi_n\to\xi_0$ in $K$, i.e. $d_\infty(\xi_n,\xi_0)\to 0$. Let us take any $\varepsilon>0$. Then for $n\in\bn$ big enough we know that $\xi_{n,k}$ is constant, i.e. $\xi_{n,k}=\xi_{0,k}$ for $k\in\supp(\xi_0)$ and   $|\xi_{n,k}|<\frac{\epsilon}{\Vert a\Vert_\cf}$   for $k\not\in\supp(\xi_0)$. Then
\[
|\Phi(a)(\xi_n)-\Phi(a)(\xi_0)|\leq \sum_{k\not\in\supp(\xi_0)}|a_k| |\xi_{n,k}|
=\sum_{k\in\supp(\xi_n)\setminus \supp(\xi_0)} |a_k| |\xi_{n,k}| 
\]
\[
\leq \frac{\epsilon}{\Vert a\Vert_\cf} \sum_{k\in\supp(\xi_n)\setminus \supp(\xi_0)} |a_k| \leq \varepsilon.
\]
This actually shows that $\Phi(a)$ is a continuous function.
\end{proof}

\subsection{Compactness criteria}

The obvious consequence of the general compactness criteria discussed in Section \ref{sec:abstract:compactness} is the following theorem:

\begin{theorem}\label{compact:fin:a-a}
 The set $A\subset \FIN(\cf)$  is relatively compact iff 
 \begin{enumerate}
\item[(\ref{compact:fin:a-a}.1)] it is bounded;
\item[(\ref{compact:fin:a-a}.2)] for any $\varepsilon>0$ there exists such $N\in\bn$ that for all $a,b\in A$
\[
\sup_{F\in\cf } \sum_{k\in F\cap[1,N]} |a_k-b_k| \geq \sup_{F\in\cf } \sum_{k\in F} |a_k-b_k| - \varepsilon.
\]
 \end{enumerate} 
\end{theorem}
\begin{proof}
It is enough to observe that a condition (\ref{compact:fin:a-a}.2) is actually the rephrased statement     that the set $A-A$ is equinormed with respect to the family of seminorms $(\Vert\cdot\Vert_{\tilde\cf})_{\tilde\cf\subset\cf\wedge |\tilde\cf|<\infty}$ in $\FIN(\phi)$ indexed by a family of finite subsets of $\cf$.
\end{proof}

This is a valid criterion, yet we are looking for something better: a description in terms of the properties of the set elements of the set $A$ rather than in terms of properties of the set of differences $A-A$. In search for this we may refer to the function space representation discussed in the previous section.

We have the following compactness criteria in the spaces of bounded functions on the set $K$ (see \cite{DS}, \cite{GKM2}):
\begin{theorem}\label{th:compact:b}
 The set $A\subset \FIN(\cf)$ (i.e. $\Phi(A)\subset B(K)$) is relatively compact iff 
 \begin{enumerate}
\item[(\ref{th:compact:b}.1)] for each $\eta\in K$ the set $\{\Phi(a)(\eta) : a\in A\}\subset\br$ is bounded;
     \item[(\ref{th:compact:b}.2)] for any $\varepsilon>0$ there exists the finite cover of the domain    $K=\bigcup_{i=1}^N K_i$ into sets $K_i\subset K$ such that 
 \[
 \forall_{\xi,\eta\in K_i} \forall_{a\in A} |\Phi(a)(\xi)-\Phi(a)(\eta)|<\varepsilon.
 \]
 \end{enumerate} 
\end{theorem}

So we can see the compactness is expressed in terms of the properties of functions $\Phi(a)$: the functions are {\em almost constant} on certain sets. 

\begin{remark}
    The condition $(\ref{th:compact:b}.1)$ in Theorem \ref{th:compact:b} is satisfied  if and only if $A\subset \FIN(\cf)$ is bounded.
\end{remark}

\begin{remark}
By Lemma \ref{fin:cont:function:space} the compactness criteria given in Theorem \ref{th:compact:b}  may be enhanced a little: we may require that the finite cover $\{K_1,\ldots,K_N\}$ be composed of subsets $K_i\subset K$ open in the metric space $(K,d_\infty)$. 
\end{remark}

It would be good to have the criteria given in Theorem \ref{th:compact:b} translated into the properties of a family $\cf$. Unfortunately, we are not able to see a nice description in these terms. Maybe it is not that natural,  as the example of the known compactness criteria in $\ell^\infty$ (cf. Theorem \ref{th:old:linf} above) shows.

\section{Discussion about equinormed sets}

Since we are not able to give a criterion which is elegant enough in the case of general $\FIN(\cf)$ spaces, we decided to follow another route: to identify the spaces for which the compactness criteria are simplified. There is a natural candidate for such a perspective: the general compactness criteria require the set $A-A$ to be equinormed as a sufficient condition for compactness. In some cases though (e.g. space $\ell^1$, as opposed to $\ell^\infty$) it appears we may replace the equinormedness of $A-A$ with the equinormedness of $A$. This is much more elegant, as the compactness is described directly in terms of the set $A$. This section contains all we were able to find out while following this route.

It is a general fact that for a bounded subset $A\subset E$ of a normed space, being precompact is equivalent to the equinormedness of the set $A-A$. But this may not be easily replaced by the equinormedness of the set $A$ itself. There is a one way implication (cf \cite[Porposition 3.13]{GKM}) saying that each relatively compact set is equinormed, but the reverse implication holds only in some spaces. The example of such space is $\ell^1 = \EXH(\cf_{<\infty}) = \FIN(\cf_{<\infty})$, the example of the space where there exist equinormed sets $A$, which are not relatively compact is $c_0=\EXH(\cf_{\leq 1})$. We are going to say more about the situations where the equinormedness of $A$ is equivalent to the relative compactness of $A$. But first the general observation comes.

\begin{proposition}\label{prop:fin:exh:equinormed}
If $\FIN(\cf)\neq\EXH(\cf)$, then there exists such $A\subset\EXH(\cf)$ that $A$ is equinormed and $A$ is not relatively compact.    
\end{proposition}
\begin{proof}
Assume $x\in\FIN(\cf)\setminus\EXH(\cf)$. This means that there exist $c>0$ and a sequence $F_n\in\cf$ such that 
\[
\sum_{k\in F_n \cap [n+1,+\infty)} |a_k| \geq c.
\]     
Let $A=\{ x\cdot\chfun_{[1,n]} : n\in \bn\}.$ As we can see, $A$ is equinormed, since there exists such $F_0\in\cf$ that
\[
\sum_{k\in F_0} |x_k| \geq \Vert x\Vert_\cf-\varepsilon 
\]
for $\varepsilon\in(0,c)$.
When $n\geq \max F_0$, then
\[
\sum_{k\in F_0} |x_k\cdot\chfun_{[1,n]}(k)| = \sum_{k\in F_0} |x_k| \geq \Vert x\Vert_\cf-\varepsilon \geq  \Vert x\cdot\chfun_{[1,n]}\Vert_\cf-\varepsilon. 
\]
On the other hand, let us take
 such a finite family $\cb \subset \cf$ that each set $B_0 \in \{ F\cap [1,\max F_0] : F\in\cf \}$ is represented by some $B\in\cb$ such that $B_0 = B\cap [1,\max F_0]$. More formally:
 $$(\tau \subseteq [1,\max F_0] \land \, \exists T \in \mathcal{F}: \; \tau \subseteq T) \implies \exists B \in \mathcal{B}: \; \tau \subseteq B.$$ Then $\cb$ is a finite family and we have 
\[
\Vert x\cdot\chfun_{[1,n]}\Vert_{\cb} = \Vert x\cdot\chfun_{[1,n]}\Vert_{\cf}  \geq \Vert x\cdot\chfun_{[1,n]}\Vert_{\cf} - \varepsilon 
\]
for all $n\leq \max F_0$. This proves that $A$ is equinormed. But $A-A$ is not equinormed, since for any finite family $\cb\subset \cf$ and  $m>n>\max \bigcup\cb$, where $m$ is big enough there is $\Vert x\cdot\chfun_{[1,m]} - x\cdot\chfun_{[1,n]}\Vert_{\cb}=0$ and 
\[\Vert x\cdot\chfun_{[1,m]} - x\cdot\chfun_{[1,n]}\Vert_{\cf}
= \Vert x\cdot\chfun_{[n+1,m]}\Vert_{\cf} \geq c.\] 
\end{proof}

\begin{example}
The Proposition \ref{prop:fin:exh:equinormed} cannot be reversed, i.e. it may happen that in case when $\FIN(\cf)=\EXH(\cf)$ there may exist equinormed and non-compact sets. Let us consider $\cf = \{\{1\}\} \cup \{ A\subset [2,+\infty) : |A|<\infty \}$. Then
\[
\Vert x\Vert_\cf = \max\{|x_1|, \sum_{n=2}^{+\infty} |x_n|\}.
\]
Obviously $\EXH(\cf)=\FIN(\cf)$ and the set $A$ consisting of all points $x_n$, $n\geq 2$ given by
\[
x_{n,k} = \begin{cases} 1 & k\in\{1,n\} \\ 0 & {\text{otherwise}}
\end{cases}
\]
is equinormed, while $A-A$ is not equinormed.
\end{example}

This shows that the question: when the conditions 
\begin{itemize}
    \item bounded set $A$ is equinormed;  
    \item bounded set $A$ is such that $A-A$ is equinormed    
\end{itemize}
are equivalent? -- is nontrivial and relevant.
The first, very natural observation is as follows:

\begin{theorem}
\label{dirtheorem}
If $\mathcal{F}$ is directed (that is, $\forall D, B \in \mathcal{F} \; \exists C \in \mathcal{F}: \; D \cup B \subseteq C$), then the equinormedness of $A \subseteq X_\mathcal{F}$ implies the equinormedness of $A-A$.
\end{theorem}
\begin{proof}
$A$ being equinormed means that, for any $\varepsilon>0$, there exists a finite subfamily $\mathcal{B} \subset \mathcal{F}$ such that $\forall x \in A: \; ||x||_\mathcal{F}-||x||_\mathcal{B}<\varepsilon$. It is possible (similarly as we did in Proposition \ref{prop:fin:exh:equinormed}) to find another finite subfamily $\mathcal{B}' \subset \mathcal{F}$ such that $$(\tau \subseteq \bigcup_\mathcal{B} \land \, \exists T \in \mathcal{F}: \; \tau \subseteq T) \implies \exists T' \in \mathcal{B}': \; \tau \subseteq T'.$$ In simple words, the sum over $\mathcal{B}$ is a finite subset of $\bn$, collecting indices that are relevant in any manner for calculating the seminorm. We can choose all sets of these indices that exist as subsets in some set from $\mathcal{F}$, and put one representative for every such set in $\mathcal{B}'$ -- this subfamily is finite.
\newline
\newline
Let us take under consideration any $z \in A-A$. Since the norm $||z||_\mathcal{F}$ is defined as a supremum, without doubt there exists a set $R \in \mathcal{F}$ such that $||z||_\mathcal{F} - \sum_{n \in R} |z_n| < \varepsilon$. Our goal is to prove that $\mathcal{B}'$ is an appropriate seminorm for $A-A$, so for this purpose we can write $R$ as a sum of two separate sets $W \cup S$, where $W=R \cap \bigcup_\mathcal{B}$ and $S=R \setminus \bigcup_\mathcal{B}$. What we need is to bound $\sum_{n \in S} |z_n|$ with respect to $\varepsilon$.
For this purpose, naturally, we can exploit the fact that $z=x-y$ for some $x, y \in A$, and also $\sum_{n \in S} |z_n| \le \sum_{n \in S} |x_n| + \sum_{n \in S} |y_n|$. Therefore it is sufficient to bound one of these sums (and without loss of generality, both) with respect
to $\varepsilon$.
\newline
\newline
Let $P \in \mathcal{B}$ be the set which realises the supremum for $||x||_\mathcal{B}$. By assumption that $\mathcal{F}$ is directed, there exists $Q \in \mathcal{F}$ such that $P \cup R \subseteq Q$, so in particular $P\cup S\subseteq Q$. Hence, $$||x||_\mathcal{B}+\sum_{n \in S} |x_n| = \sum_{n \in P} |x_n|+\sum_{n \in S} |x_n| \le \sum_{n \in Q} |x_n| \le ||x||_\mathcal{F} \le ||x||_\mathcal{B}+\varepsilon,$$  so that indeed $\sum_{n \in S} |x_n| \le \varepsilon$. Finally, $$||z||_{\mathcal{B}'} \ge \sum_{n \in W} |z_n| \ge \sum_{n \in R} |z_n|-2\varepsilon > ||z||_\mathcal{F}-3\varepsilon.$$ Given that, for any $\varepsilon>0$, there exists a finite subfamily $\mathcal{B}' \subset \mathcal{F}$ such that $\forall z \in A-A: \; ||x||_\mathcal{F}-||x||_{\mathcal{B}'}<3\varepsilon$, $A-A$ is equinormed. This completes the proof.
\end{proof}

\begin{remark}
The directedness of the family $\cf$ is actually the very strong assumption. The only directed and hereditary family is the one given by all finite subsets of $\bn$: all directed families $\cf$ have the same hereditary closure equal to $\cf_{<\infty}$, hence for each such family the induced norm is equal to the norm of $\ell^1$. Still we decided to present this fact as it is a natural starting point for further investigations.
\end{remark}

This theorem can be made stronger by loosening the condition that $\mathcal{F}$ should be directed. Because this point discusses only the underlying structure of the combinatorial Banach space $\FIN(\cf)$, and not its contents, it demonstrates why these spaces are called \emph{combinatorial} and thus reveals an intriguing connection between functional analysis and combinatorics. The stronger theorem will be presented in the continuation of this section and the combinatorial considerations in the next section. At present, please note that the following definition is purely set-theoretical in its heart, spanning a sequence over $C$ is only accidental to it: 

\begin{definition}[Covering by weight]\label{def:cov:by:weight}
For any $0 \le \alpha \le 1$ and any family $\mathcal{E}(C)$ of nonempty subsets of a set $C$, we will say that a family $\mathcal{E}(C)$ belongs to $W_\alpha(C)$ (or \emph{covers at least $\alpha$ of $C$ by weight}) when, for any real sequence $x$ indexed over $C$, $$\exists C' \in \mathcal{E}(C) \; \sum_{i \in C'} |x_i| \ge \alpha\sum_{i \in C} |x_i|.$$
\end{definition}

While for the definition to work correctly it is enough to assume that $C$ is countable and the sum of $x$ over $C$ is finite, in all further considerations we will use finite sets $C$.

\begin{remark}
Please note that this notion is, in a way, monotone with respect to $\alpha\in[0,1]$:
\[
\alpha'\leq\alpha \implies W_\alpha(C)\subseteq W_{\alpha'}(C).
\]
We may also easily observe that if $\mathcal{E}(C)\in W_1(C)$, then $C\in\mathcal{E}(C)$ and any nonempty family $\mathcal{E}(C)$ belongs to $W_0(C)$.
\end{remark}

\begin{example}
Assume $C\subset\bn$ satisfies $\#C=N$ and 
\[
\mathcal{E}(C) = \{ B\subset C : \#B = N-1 \}.
\]
Then $\mathcal{E}(C)$ covers at least $\alpha = 1 - \frac{1}{N}$ of $C$ by weight.
\end{example}

Now we can proceed to discuss a more complex definition that should allow to classify the combinatorial Banach spaces almost completely with respect to whether the equinormedness of $A \subset \FIN(\cf)$ is a sufficient condition for compactness: 

\begin{definition}[Classes of directedness]
For a family $\mathcal{F}$ of finite sets in $\bn$ and any $N \in \bn$, let us denote the family of all nonempty sets of the form $$B_N=B \cap [N+1,+\infty): B \in \mathcal{F}$$ as $\mathcal{F}_N$. $\mathcal{F}$ will be said to belong to the \emph{$(\lambda, \gamma)$-class of directedness} when $$\forall D \in \mathcal{F} \; \exists N \in \bn \; \forall B \in \mathcal{F}_N:$$$$\exists \mathcal{E}(D) \in W_{1-\lambda}(D), \mathcal{E}(B) \in W_\gamma(B) \; \forall D' \in \mathcal{E}(D), B' \in \mathcal{E}(B) \; \exists C \in \mathcal{F}:$$$$D' \cup B' \subseteq C.$$ 
\end{definition}

From the fact that $D' \in \mathcal{E}(D), B' \in \mathcal{E}(B)$ are both nonempty it follows that $D \cap C$ and $B \cap C$ are also both nonempty. This allows to avoid the situation that every family would belong to the $(0, 0)$-class.

Ideologically speaking, the point of this definition is to weaken the notion of directedness while retaining its aspects that are central for our considerations. Instead of demanding that there exists a set $C \in \mathcal{F}$ containing both $D$ and $B$, we allow that it contains \emph{almost all} of any chosen set $D$ still along with \emph{a significant fragment} of another sufficiently distant set $B$. One can easily see that the behaviour of classes of directedness is uniform, in that if $\mathcal{F}$ belongs to the $(\lambda, \gamma)$-class, it also belongs to the $(\lambda', \gamma')$-class for any $\lambda'\ge\lambda, \gamma'\le\gamma$.

\begin{theorem}
If $\mathcal{F}$ belongs to the $(0,c)$-class of directedness for some $c>0$, then the equinormedness of $A \subseteq \FIN(\mathcal{F})$ implies the equinormedness of $A-A$.
\end{theorem}
\begin{proof}
Assume $A\subseteq\FIN(\cf)$ is equinormed (as usual with respect to the family of seminorms given by the finite subsets of $\cf$).
For any $F \in \mathcal{F}$ we can choose such $N \in\bn$ that for any $B \in \mathcal{F}_N$ there exists $C \in \mathcal{F}$ ($F \subset C$) containing at least $c$ of $B$ by weight. Repeating the proof of Theorem \ref{dirtheorem}, we have a finite subfamily $\mathcal{B} \subset \mathcal{F}$ such that $\forall x \in A: \; ||x||_\mathcal{F}-||x||_\mathcal{B}<\varepsilon$. We can adopt as the new $N$ the largest of values of $N$ for $F \in \mathcal{B}$. We change the way $R$ is split into two sets so that $W=R \cap [1,N]$ and $S=R \cap [N+1,+\infty)$, intending to show that the sum $\sum_{n \in S} |x_n|$ can be bounded with respect to some function of $\varepsilon$.
\newline
\newline
Like previously, let $P \in \mathcal{B}$ be the set which realises the supremum for $||x||_\mathcal{B}$. By assumption that $\mathcal{F}$ belongs to the $(0,c)$-class of directedness, there exists $Q \in \mathcal{F}$ and $S' \subset S$ such that $P \cup S' \subseteq Q$ and $\sum_{i \in S'} |x_i| \ge c\sum_{i \in S} |x_i|$. Hence, $$||x||_\mathcal{B}+c\sum_{n \in S} |x_n| \le \sum_{n \in P} |x_n|+\sum_{n \in S'} |x_n| \le \sum_{n \in Q} |x_n| \le ||x||_\mathcal{F} \le ||x||_\mathcal{B}+\varepsilon,$$ so that indeed $\sum_{n \in S} |x_n| \le \frac{\varepsilon}{c}$. Defining the finite subfamily $\mathcal{B}' \subset \mathcal{F}$ on the basis of $[1,N]$ rather than $\bigcup_\mathcal{B}$, we complete the proof.
\end{proof}

\begin{theorem}
If $\mathcal{F}$ belongs to the classes of directedness  $(\lambda, c')$ arbitrarily close to  $(0,c)$ (that is, there exists a sequence of such classes to which $\mathcal{F}$ belongs and its limit is by coordinates $(0,c)$), then the equinormedness of $A \subseteq \FIN(\cf)$ implies the equinormedness of $A-A$.
\end{theorem}
\begin{proof}
Let us fix $\lambda\in(0,1)$ and $\delta\in(0,c)$ such that $\cf$ belongs to $(\lambda,c-\delta)$ class of directedness. Then for any $F \in \mathcal{F}$ we can choose such $N \in\bn$ that for any $B \in \mathcal{F}_N$ there exists $C \in \mathcal{F}$ containing at least $1-\lambda$ of $F$ and at least $c-\delta$ of $B$ by weight. Proceeding to the crucial inequality same as above, we arrive at $$(1-\lambda)||x||_\mathcal{B}+(c-\delta)\sum_{n \in S} |x_n| \le ||x||_\mathcal{B}+\varepsilon,$$ that is $\sum_{n \in S} |x_n| \le \frac{\lambda||x||_\mathcal{B}+\varepsilon}{c-\delta}$. Thus, choosing sufficiently small values of $\lambda$ and $\delta$, we can uniformly restrict the sum over $S$ dependent on $\varepsilon$ -- say, $\sum_{n \in S} |x_n| \le \frac{2\varepsilon}{c}$ -- what was to be proven.
\end{proof}

\begin{theorem}
If $\mathcal{F}$ does not belong to the $(\lambda, 0)$-class of directedness for some $\lambda>0$, then the equinormedness of $A \subseteq \FIN(\cf)$ does not imply the equinormedness of $A-A$.
\end{theorem}
\begin{proof}
In this situation, there exists a set $F \in \mathcal{F}$ such that for any $N \in\bn$ we can find nonempty $B \in \mathcal{F}_N$ with the property that any  $C \in \mathcal{F}$ ($B \cup C \neq \emptyset$) does not contain $1-c$ of $F$ by weight. Therefore,  we can assign values to elements of a sequence $x$ in such a manner that $\sum_{i \in F} |x_i|=1$, $\sum_{i \in B} |x_i|=c$ and $||x||_\mathcal{F}=1$. Composing the set $A$ from such sequences with increasingly distant sets $B$, we assure that $A$ is equinormed but $A-A$ is not. 
\end{proof}

\noindent
\textbf{Open question}: what can be said about the relation between the equinormedness of $A \subseteq \FIN(\cf)$ and the equinormedness of $A-A$ when $\mathcal{F}$ belongs to the classes of directedness arbitrarily close to $(0, 0)$?

\begin{example}
We will denote the interval $[10^{i-1}+1, 10^i]$ as $D_i$; as an exception, let $D_1=[1,10]$. Let $\mathcal{D}ec$ be a family of finite sets in $\bn$ which for some $i$ include as a subset the whole interval $D_i$ (called its \emph{principal interval}) and, for finitely many values of $j \ge 2i$, $\frac{1}{9}$ of the interval $D_j$ (in other words, $10^{j-1}$ elements from between $10^{j-1}+1$ and $10^j$). Clearly, $\mathcal{D}ec$ is not directed, as a sum of two sets from it can easily contain two whole intervals, which is never possible for a single set. However, for any given $D \in \mathcal{D}ec$ with a principal interval $D_i$, we can find such natural $k$ that $k \ge 2i$ and also $D_k$ lies further than any of the intervals featured in $D$; let $N=10^k+1$. Now, for any $B \in \mathcal{D}ec_N$, there exists some $C \in \mathcal{D}ec$ such that $D \subset C$, $C$ shares with $B$ all its secondary intervals, and $C$ contains any chosen $\frac{1}{9}$ of $B$'s principal interval. From here it follows that $C$ contains at least $\frac{1}{9}$ of $B$ by weight, and hence $\mathcal{D}ec$ belongs to the $(0, \frac{1}{9})$-class of directedness. The equinormedness of $A \subseteq \FIN(\mathcal{D}ec)$ implies the equinormedness of $A-A$.
\end{example}

Given that in $\mathcal{D}ec$ there exists a set containing the whole interval $D_1$ and any chosen $\frac{1}{9}$ of as many next intervals as one wishes, for any $x \in \FIN(\mathcal{D}ec)$ there occurs $||x||_{\mathcal{D}ec} \ge \frac{1}{9}\sum_{i=1}^\infty |x_i|$. Thus the combinatorial norm of $\mathcal{D}ec$ is equivalent to the norm of $\ell_1$.

\begin{example}
Let $\mathcal{D}ed$ be a family of finite sets in $\bn$ which for some $i$ include as a subset the whole interval $D_{i_0}$ (called its \emph{principal interval}) and, for finitely many values of $i_{k+1} \ge 2i_k$, $\frac{1}{9}$ of the interval $D_{i_{k+1}}$ (in other words, $10^{i_{k+1}-1}$ elements from between $10^{i_{k+1}-1}+1$ and $10^{i_{k+1}}$). $\mathcal{D}ed$ is not directed for the exact same reason as $\mathcal{D}ec$. However, for any given $D \in \mathcal{D}ed$ with a principal interval $D_{i_0}$ and secondary intervals numbered to $D_{i_m}$, we can pick $N=10^{2i_m}+1$. Now, just like in case of $\mathcal{D}ec$, for any $B \in \mathcal{D}ed_N$, there exists some $C \in \mathcal{D}ed$ such that $D \subset C$, $C$ shares with $B$ all its secondary intervals, and $C$ contains any chosen $\frac{1}{9}$ of $B$'s principal interval. From here it follows that $C$ contains at least $\frac{1}{9}$ of $B$ by weight, and hence $\mathcal{D}ed$ belongs to the $(0, \frac{1}{9})$-class of directedness. The equinormedness of $A \subseteq \FIN(\mathcal{D}ed)$ implies the equinormedness of $A-A$.
\end{example}

For any natural $c$, let the sequence $x(c)$ consist of $\frac{1}{c}$ spread evenly over each of the intervals $D_c$ to $D_{2c-1}$. Then $\sum_{i=1}^\infty |x(c)_i|=1$ but $||x||_{\mathcal{D}ed}=\frac{1}{c}$. Thus the combinatorial norm of $\mathcal{D}ec$ is not equivalent to the norm of $\ell_1$.
\newline\newline
This is an important observation that in case of combinatorial Banach spaces that are not yielded by directed set families but still have the property that a bounded and equinormed set in them is relatively compact, their norms can but do not have to be equivalent to the norm of $\ell_1$.

\section{Combinatorics}

The beginning of our combinatorial pursuit is to find an alternative language and see the inner structure of combinatorial spaces as something else than the families of subsets of $\bn$. It is known that the structure of some of these spaces can be expressed as a graph (i.e. the approach of mapping the cliques in a possibly infinite graph to the subsets forming a combinatiorial space as described in \cite{JachimekPhD}). More generally though, the natural language to describe a set with a family of its subsets in terms of combinatorics and graph theory is by a hypergraph.

\begin{definition}[Hypergraph]
An (undirected) \emph{hypergraph} is a pair $(V, E)$ where $V$ is a set whose elements are called \emph{vertices} and $E$ is a set of subsets of $V$ whose elements are called \emph{hyperedges}. Therefore, it is a generalization of a graph in which an edge is permitted to connect any number of vertices.
\end{definition}

\begin{definition}[Incidence matrix]
An incidence matrix uniquely representing a hypergraph with vertices $v_1, ..., v_n$ and hyperedges $e_1, ..., e_m$ is a $n \times m$ matrix with ones at the cell $(i, j)$ if $v_i$ is connected by the edge $e_j$ and zeros otherwise.
\end{definition}

\begin{definition}[Dual of a hypergraph]
The dual of a hypergraph represented by an incidence matrix $H$ is the hypergraph represented by the incidence matrix $H^T$.
\end{definition}

It is necessary to know that in the abstract structure of hypergraphs there exists a deep symmetry between the vertices and the hyperedges, as demonstrated by the following observation:

\begin{remark}
Any hypergraph can be equivalently described as a bipartite graph with the vertices on one side representing the original vertices, the vertices on the other side representing the hyperedges, and the edges denoting which vertices are connected by a given hyperedge. The dual of any given hypergraph is described by its same bipartite graph with the roles of its sides reversed.
\end{remark}

Many classical problems from the world of graph theory have a natural extension to hypergraphs. For us this is relevant because we are interested in finding analogies between the concepts related to solving these problems and the classes of combinatorial spaces. First and foremost, these topics worth discussing are matching and vertex cover in hypergraphs, because of their noted connections to the set cover problem and by extension to properties of set families.

\begin{definition}[Matching]
A \emph{matching} in a hypergraph is a set of hyperedges such that no two of them have a vertex they would both contain (they are all disjoint).
\end{definition}

\begin{definition}[Vertex cover]
A \emph{vertex cover} in a hypergraph is a set of vertices such that every hyperedge contains at least one of them.
\end{definition}

\begin{definition}[Set cover]
A \emph{set cover} in a family of subsets of a given set is a collection of these subsets such that their union is the whole set. Equivalently, a set cover in a hypergraph is a set of hyperedges such that every vertex is contained in at least one of them.
\end{definition}

The most important characterization of a hypergraph with regard to any of these concepts is its so-called \emph{number}, which denotes the smallest or largest (as relevant) possible cardinality of a set from the definition. Thus the matching number of a hypergraph is the maximal size of a set of its pairwise disjoint hyperedges, the set cover number of a hypergraph is the minimal size of a set of its hyperedges covering all vertices, etc.

These concepts have their \emph{fractional} variants, where the modification is that instead of each vertex or edge being included or not in the configuration, they have weights from the $[0,1]$ interval assorted to them, while keeping the analogous restrictions on their sums. It is easiest to describe them in the language of incidence matrices:

\begin{definition}[Fractional matching]
A \emph{fractional matching} in a hypergraph represented by an incidence matrix $H_{n,m}$ is a set of weights $w_1, ..., w_m$ such that $$\forall k \in \{1,...,n\} \; \sum_{i=1}^m H_{k,i}w_i \le 1.$$
\end{definition}

\begin{definition}[Fractional vertex cover]
A \emph{fractional vertex cover} in a hypergraph represented by an incidence matrix $H_{n,m}$ is a set of weights $w_1, ..., w_n$ such that $$\forall k \in \{1,...,m\} \; \sum_{i=1}^n H_{i,k}w_i \ge 1.$$
\end{definition}

\begin{definition}[Fractional set cover]
A \emph{fractional set cover} in a hypergraph represented by an incidence matrix $H_{n,m}$ is a set of weights $w_1, ..., w_m$ such that $$\forall k \in \{1,...,n\} \; \sum_{i=1}^m H_{k,i}w_i \ge 1.$$
\end{definition}

The definitions of graph numbers also extend naturally to \emph{fractional numbers}, which mean simply the upper or lower bounds on the sum of weights from the definitions.

By now it should be easy to notice that set cover and vertex cover, as well as their fractional variants, are symmetrical by the property of hypergraph dualism: a set cover in a hypergraph is a vertex cover in its dual and vice versa, so in particular the (fractional) set cover number of one is equal to the (fractional) vertex cover number of another. The following theorem, proven by the means of linear programming duality, gives another important connection between these concepts \cite{Matching}:

\begin{theorem}[Füredi, 1981]\label{th:furedi}
    The fractional matching number of any hypergraph is equal to its fractional vertex cover number.
\end{theorem}

A careful reader may have taken note that a fourth concept can be defined, complementary to the previous three and forming with them a $2 \times 2$ square of symmetries. While this fourth problem -- named the \emph{vertex packing}, \emph{independent set}, or \emph{anticlique} -- is well known for regular graphs, it does not seem to have attracted as much attention with regard to hypergraphs as the previous three:

\begin{definition}[Anticlique]
An \emph{anticlique} in a hypergraph is a set of vertices such that no two of them have a hyperedge that would connect them (contain both of them).
\end{definition}

However, it does have noted connections to central topics in modern mathematics like Szemerédi's theorem on arithmetic progressions \cite{Anticlique}. It is much harder to find any mention of its fractional variant, which completes the aforementioned symmetry and happens to be useful for our classification of combinatorial Banach spaces:

\begin{definition}[Fractional anticlique]
A \emph{fractional anticlique} in a hypergraph represented by an incidence matrix $H_{n,m}$ is a set of weights $w_1, ..., w_n$ such that $$\forall k \in \{1,...,m\} \; \sum_{i=1}^n H_{i,k}w_i \le 1.$$
\end{definition}

\begin{theorem}
For any nonempty set $A$ and family of its nonempty subsets $\mathcal{E}(A)$, let us denote by $u(A, \mathcal{E}(A))$ the supremum on values of $\alpha$ such that $\mathcal{E}(A)$ covers at least $\alpha$ of $A$ by weight as given in Definition \ref{def:cov:by:weight}, and by $\mu^*(A, \mathcal{E}(A))$ the fractional anticlique number of the hypergraph created by mapping the elements of $A$ to vertices and the elements of $\mathcal{E}(A)$ to hyperedges. When, as we regularly assume, the union over the family $\mathcal{E}(A)$ covers the whole $A$ (and thus $u(A, \mathcal{E}(A))>0$), there occurs that $u(A, \mathcal{E}(A)) \cdot \mu^*(A, \mathcal{E}(A))=1$.
\end{theorem}
\begin{proof}
The notion that that $\mathcal{E}(A)$ covers at least $\alpha$ of $A$ by weight means the following: allocating weights summing to $1$ to the elements of $A$ guarantees that at least one subset from $\mathcal{E}(A)$ has weight greater or equal to $\alpha$. By proportionality, it happens if and only if allocating weights summing to $\frac{1}{\alpha}$ to the elements of $A$ guarantees that at least one subset from $\mathcal{E}(A)$ will be weighed to at least~$1$. 

Therefore $\frac{1}{u(A, \mathcal{E}(A))}$ is the lower bound on the sums of weights having this last property. The fractional anticlique number $\mu^*(A, \mathcal{E}(A))$, on the other hand, is by definition the upper bound on the sums of weights where all subsets from $\mathcal{E}(A)$ are weighed to at most $1$. Yet let us take a number $\beta$ which is below one of these two bounds. Then there exists a configuration of weights assigned to the elements of $A$ and summing to $\beta$ such that each one of the subsets in the family $\mathcal{E}(A)$ is weighed at strictly less than $1$, which implies that $\beta$ is also below the other bound. This completes the proof.
\end{proof}

\begin{theorem}
The fractional anticlique number of any hypergraph is equal to its fractional set cover number.
\end{theorem}
\begin{proof}
As we already know, the fractional set cover number of a hypergraph is equal to the fractional vertex cover number of its dual, the fractional anticlique number of a hypergraph is equal to the fractional matching number of its dual, and the equality of those two follows from Theorem \ref{th:furedi} (applied to that dual).
\end{proof}

In this way we have finally managed to rewrite the definition of covering by weight in well-known terms: any set $A$ is covered by weight by a family of its subsets $\mathcal{E}(A)$ to the level of $\frac{1}{\nu^*(A, \mathcal{E}(A))}$, where $\nu^*$ is traditionally used to denote the fractional set cover number. This is the effect we sought -- demonstrating the usefulness of describing the inner structure of combinatorial Banach spaces by means of hypergraph theory. For additional benefit of the reader, we will also show this proof as yielded directly by linear programming methods:


Let us first repeat and slightly rewrite the formulations of our two problems, of which one is closely related to the concept of {\em cover by weight} appearing in Definition \ref{def:cov:by:weight} and the other is very well known in combinatorics.

Both notions will be defined for a finite set $\cs$ such that $|\cs|=N\in\bn$, and the family $\cp=\{P_1,...,P_M\}$ of subsets of $\cs$ covering $\cs$, i.e. such that $\cs = \bigcup_{P\in\cp} P$. 

\begin{definition}
The Cover by Weight problem for a finite set $\cs=\{s_1,...,s_N\}$ and a family $\cp=\{P_1,...,P_M\}$ covering $\cs$ is the problem of finding such value $\alpha\in\br$ that for any sequence of nonnegative weights $(x_1,...,x_N)$ of elements of $\cs$ such that $x_1+...+x_N=\alpha$ there exists at least one subset $P\in\cp$ containing elements of total weight at least $1$, i.e. $\sum_{j\in P}x_j\geq 1$. We will denote the infimum of all such values of $\alpha$ by $\mu^*(\cs,\cp)$.  
\end{definition}

\begin{definition}
The Fractional Set Cover problem for a finite set $\cs=\{s_1,...,s_N\}$ and a family $\cp=\{P_1,...,P_M\}$ covering $\cs$ is the
problem of finding such value $\beta\in\br$ that there exists a sequence of nonnegative weights $(y_1,...,y_M)$ of subsets $P\in\cp$, such that for each element $s_i\in\cs$ the sum of weights of sets $P$ containing $s_i$ is at least $1$. We will denote the infimum of all such values of $\beta$ by $\nu^*(\cs,\cp)$. 
\end{definition}

It appears that the values $\mu^*(\cs,\cp)$ and $\nu^*(\cs,\cp)$ may be expressed as linear programming problems.

\begin{proposition}\label{prop:lp:problems}
The following equalities hold:
\begin{itemize}
    \item[(i)] $\mu^*(\cs,\cp) = \max\{ \sum_{i=1}^N x_i : \forall_{P\in \cp} \sum_{j\in\ P} x_j \leq 1 \wedge \forall_{j\in\{1,...,N\}} x_j\geq 0 \} = $
    \[
    \max\{ \sum_{i=1}^N x_i : \forall_{P\in \cp} \sum_{j=1}^N x_j\chfun_P(j) \leq 1 \wedge \forall_{j\in\{1,...,N\}} x_j\geq 0 \};
    \]
    \item[(ii)] $\nu^*(\cs,\cp) = \min\{ \sum_{j=1}^M y_j : \forall_{s\in\cs} \sum_{j : s\in P_j} y_j \geq 1 \wedge \forall_{j\in\{1,...,N\}} y_j\geq 0\}=$ 
    \[
    \min\{ \sum_{j=1}^M y_j : \forall_{s\in\cs} \sum_{j=1}^M y_j\chfun_{P_j}(s) \geq 1 \wedge \forall_{j\in\{1,...,N\}} y_j\geq 0\}.
    \]
\end{itemize}
\end{proposition}

The proof is obvious and will be omitted.

The optimization problems given in Proposition \ref{prop:lp:problems} may be treated as linear programming problems with
\begin{itemize}
    \item[($\mu^*$)] maximize $c^T x$ subject to $Ax\leq b$ and $x\geq 0$, $x\in\br^N$, where $c=[1,\ldots,1]^T\in\br^N$, $b=[1,\ldots,1]^T\in\br^M$ and $A$ is $M\times N$ matrix $A=[a_{i,j}]$, where $a_{i,j}=\chfun_{P_i}(j)$.
    \item[($\nu^*$)] minimize $\tilde c^T y$ subject to $\tilde A y \geq \tilde b$ and $y\geq 0$, $y\in\br^M$, where  $\tilde c=[1,\ldots,1]^T\in\br^M$, $\tilde b=[1,\ldots,1]^T\in\br^N$ and $\tilde A$ is $N\times M$ matrix $\tilde A=[\tilde a_{i,j}]$, where $\tilde a_{i,j}=\chfun_{P_j}(i)$.
\end{itemize}

\begin{proposition}
     \[
     \mu^*(\cs,\cp)=\nu^*(\cs,\cp).
     \]
\end{proposition}
\begin{proof}As we can see $\tilde c=b$, $\tilde b = c$ and $\tilde A=A^T$. Hence the linear programming problem $\mu^*$ is dual to the linear programming problem $\nu^*$. This means that, by the strong duality theorem (cf. \cite[Section 6]{LinearProgramming}), there is $\max c^T x = \min\tilde c^T y$ what completes the proof. 
\end{proof}

This is a meaningful result because it shows that the conditions of compactness in combinatorial Banach spaces are directly connected with the much more recognizable topic of set cover problem. It is one of twenty-one Karp's NP-complete problems and, together with its fractional variant, played a central role in modern research on approximation algorithms \cite{Approx}.

\section*{Acknowledgement}
The work of Jacek Gulgowski was supported by the project "Center for Trustworthy Artificial Intelligence for Life Sciences" implemented under the International Research Agendas programme of the Foundation for Polish Science co-financed by the European Union under the European Funds for Smart Economy 2021-2027 (FENG).    

Both authors are grateful to Anna Pelczar-Barwacz and Piotr Borodulin-Nadzieja for inspiring discussions.

\bibliographystyle{amsplain}
\bibliography{references.bib}

\end{document}